\documentclass[12pt]{amsart}

\usepackage{import}
\usepackage{preamble/Pakete}
\usepackage{preamble/Befehle}
\begin{document}
\raggedbottom


\title[Condensed configurations and valuations]
{Condensed configurations and valuative matroid invariants}
\author{Jens Niklas Eberhardt}
\address{Institut f\"ur Mathematik, Johannes Gutenberg-Universit\"at Mainz,
  Staudingerweg 9, 55128 Mainz, Germany}
\email{mail@jenseberhardt.com}
\date{September 16, 2026}


\begin{abstract}
Condensed configurations are compact incidence data obtained by grouping
the cyclic flats of a matroid.  We show that their inverse incidence matrices
give explicit Schubert expansions and hence determine every
valuative or covaluative matroid invariant.  For the
extended binary Golay matroid, this unexpectedly produces non-real-rooted
Kazhdan--Lusztig and \(Z\)-polynomials.  To our knowledge, the latter is
the first counterexample in the literature.
\end{abstract}

\maketitle


\section{Introduction}

The \emph{condensed configuration} of a matroid groups its cyclic flats
into blocks and records the common rank and cardinality of each block
together with the incidence numbers between
blocks~\cite{Eberhardt2014}.  For highly symmetric matroids, condensation by
automorphism orbits can be extremely small.
The extended binary Golay matroid has \(2\,047\,118\) flats and
\(39\,516\) cyclic flats, whereas its \(M_{24}\)-orbit condensation is
a labelled \(6\times6\) integer matrix.

We proved that this data determines the Tutte
polynomial~\cite[Theorem~5.1]{Eberhardt2014}.  Bonin and Kung proved that
the full configuration, that is, the lattice of all cyclic flats labelled
by rank and cardinality, determines the numbers of maximal flags of flats
with prescribed cardinalities and Derksen's
\(\Ginv\)-invariant~\cite[Theorem~7.3]{BoninKung2018}.  In this paper, we
prove the following (Theorem~\ref{thm:Schubert-expansion} and
Corollary~\ref{cor:covaluative-expansion}).

\begin{theorem*}
A condensed configuration of a finite matroid without loops or coloops
determines every valuative or covaluative matroid invariant.
\end{theorem*}

This includes counts of flats and flags,
the Tutte and Ehrhart polynomials, Derksen's \(\Ginv\)-invariant,
Speyer's \(g\)-polynomial, and the Kazhdan--Lusztig, \(Z\)-, and Chow
polynomials.

Let us explain the idea of the proof.  A matroid is called Schubert if
its cyclic flats form a chain.  A Schubert expansion expresses
the value of a valuative invariant on a matroid as an integer linear
combination of its values on Schubert matroids.  Derksen and Fink
established such expansions for matroid
polytopes~\cite[Theorem~4.2]{DerksenFink2010}.
The cyclic-flat expansions of Hampe and Ferroni~\cite{Hampe2017,Ferroni2023}
have coefficients which, by Panzer's factorization~\cite{Panzer2025},
are products of signed M\"obius values.  We show that the inverse condensed
incidence matrix records sums of these M\"obius values over condensation
blocks.  This allows us to sum the coefficients of chains with the same
block sequence one step at a time, obtaining an expansion computed
entirely from the condensed configuration.

The rank-\(3\) matroid associated with a Steiner \(2\)-design has a
\(3\times3\) condensed incidence matrix.  Our formula expresses every
valuative invariant of this matroid as an integer linear combination
of its values on just two Schubert matroids.

For the extended binary Golay matroid, a \(6\times6\) condensed incidence
matrix gives nine Schubert terms.  This allows us, for example, to
compute its \(P\)- and \(Z\)-polynomials efficiently.  Much to our
surprise, neither polynomial is real-rooted, contrary to the conjectures
in~\cite[Conjecture~3.2]{GPY2017}
and~\cite[Conjecture~5.1]{PXY2018}, respectively.  Independently,
Cheng and Liu prove the stronger statement
for \(P\)-polynomials that representable examples over every finite
field need not even be unimodal~\cite{ChengLiu2026}.  The Golay example
appears to be the first counterexample in the literature for \(Z\),
whose coefficients are nevertheless always unimodal by
\(\gamma\)-positivity~\cite[Theorem~1.8]{FerroniMatherneStevensVecchi2024}.


\section{Condensed configurations}

We assume familiarity with basic matroid theory.  Throughout, \(M\) is
a finite matroid without loops or coloops, with ground set \(E\), rank
\(r=r(M)\), and lattice of cyclic flats \(\Zcyc(M)\).

A flat \(X\) is \emph{cyclic} if \(M|X\) has no coloops.

\begin{definition}[Condensed configuration]
The \emph{configuration} of \(M\) is the lattice \(\Zcyc(M)\), labelled
by rank and cardinality.  A \emph{condensation} is a partition \(\Pcal\)
of \(\Zcyc(M)\) into blocks of constant rank and cardinality.
We require that, for any two blocks \(B,C\), every flat in \(C\)
contains the same number of flats from \(B\).  For \(Y\in C\), we
denote this common number by
\[
 A_{\Pcal}(B,C)
 =\#\{X\in B\mid X\subseteq Y\}.
\]
For \(X\in B\), write
\[
 s_B=|X|,\qquad \rho_B=\rk(X).
\]
The block labels \((s_B,\rho_B)\) and the numbers
\(A_{\Pcal}(B,C)\) are the \emph{condensed configuration}.
\end{definition}

The singleton partition gives the full configuration.  A subgroup of
\(\operatorname{Aut}(M)\) gives a smaller condensation by its orbits.
The six-block Golay condensation below is of this form.  When there are
few cyclic-flat orbits, the labelled orbit matrix can be far smaller than
both the cyclic-flat lattice and the full flat lattice.

Write \(B\leq C\) when \(A_{\Pcal}(B,C)>0\).  This is a partial order
on \(\Pcal\).  Let \(\zero=\{\varnothing\}\) and \(\one=\{E\}\).  Then
\(|B|=A_{\Pcal}(B,\one)\).


\section{Schubert coordinates from inverse incidence}

Let \(\mu_{\Zcyc}\) be the M\"obius function of the cyclic-flat lattice.
If \(A_{\Pcal}(B,C)>0\), then \(\rho_B\leq\rho_C\), with equality only
when \(B=C\).  Moreover, \(A_{\Pcal}(B,B)=1\).  Thus \(A_{\Pcal}\) is
upper unitriangular when the blocks are ordered by rank.  Its inverse is
exactly the block sum of \(\mu_{\Zcyc}\).

\begin{lemma}[Condensed M\"obius function]\label{lem:condensed-moebius}
For \(Y\in C\),
\[
 (A_{\Pcal}^{-1})_{BC}
 =
 \sum_{\substack{X\in B\\X\subseteq Y}}
 \mu_{\Zcyc}(X,Y).
\]
In particular, the sum on the right depends only on the block \(C\) of
\(Y\).
\end{lemma}

\begin{proof}
Fix \(Y\in C\), and for every block \(D\) put
\[
 m_D(Y)=
 \sum_{\substack{Z\in D\\Z\subseteq Y}}
 \mu_{\Zcyc}(Z,Y).
\]
Write \(m(Y)=(m_D(Y))_{D\in\Pcal}\).
For every block \(B\), the condensation property gives
\[
\begin{aligned}
 \bigl(A_{\Pcal}m(Y)\bigr)_B
 &=
 \sum_{D\in\Pcal} A_{\Pcal}(B,D)m_D(Y)\\
 &=
 \sum_{\substack{X\in B\\X\subseteq Y}}
 \ \sum_{X\subseteq Z\subseteq Y}\mu_{\Zcyc}(Z,Y).
\end{aligned}
\]
By the defining identity for the M\"obius function,
\[
 \sum_{X\subseteq Z\subseteq Y}\mu_{\Zcyc}(Z,Y)
 =\begin{cases}1,&X=Y,\\0,&X\ne Y.\end{cases}
\]
Thus only \(X=Y\) can contribute to the outer sum.  Since \(Y\in C\),
this happens exactly when \(B=C\).  Hence
\(m_B(Y)=(A_{\Pcal}^{-1})_{BC}\), as required.
\end{proof}

A \emph{Schubert matroid} is precisely a matroid whose lattice of cyclic
flats is a chain~\cite[Definition~2.6]{FerroniSchroeter2024}.  Other
sources call these nested, shifted, or generalized Catalan
matroids~\cite{Hampe2017,Ferroni2023}.  A strict block chain
\[
 \mathbf B=(\zero=B_0<B_1<\cdots<B_k=\one)
\]
is realized by an actual cyclic-flat chain.  Indeed, one may start with
\(E\) and choose successively downwards, using
\(A_{\Pcal}(B_{i-1},B_i)>0\) at each step.  Its cardinality--rank labels
determine a Schubert matroid up to isomorphism.  Let
\(\Sch_{\mathbf B}\) be the unique Schubert matroid whose cyclic-flat
chain has labels \((s_{B_i},\rho_{B_i})\).  Different chains of actual
cyclic flats having the same block sequence give isomorphic Schubert
matroids.  When the labels are to be displayed we also write
\[
 \Sch_{\mathbf B}
 =\Sch\bigl((s_{B_0},\rho_{B_0}),\ldots,(s_{B_k},\rho_{B_k})\bigr).
\]

Informally, a valuative invariant satisfies inclusion--exclusion whenever
a matroid base polytope is subdivided into matroid base polytopes.
We call an isomorphism-invariant function on matroids a
\emph{valuative matroid invariant} if its restriction to every fixed
ground set is a valuation.  Covaluative matroid invariants are defined
similarly.

\begin{theorem}[Condensed Schubert expansion]\label{thm:Schubert-expansion}
Let \(\Phi\) be a valuative matroid invariant with values in an abelian
group.  Then
\begin{equation}\label{eq:Schubert-expansion}
 \Phi(M)=
 \sum_{\zero=B_0<\cdots<B_k=\one}
 \left(
 \prod_{i=1}^{k}
   -(A_{\Pcal}^{-1})_{B_{i-1},B_i}
 \right)
 \Phi(\Sch_{\mathbf B}).
\end{equation}
The condensed configuration therefore gives the Schubert coordinates
of \(M\) after isomorphic Schubert matroids are identified.
\end{theorem}

\begin{proof}
Let \(\mathsf C(\Zcyc)\) be the poset of cyclic-flat chains containing
\(\varnothing\) and \(E\), ordered by inclusion and with an artificial
top \(\widehat1\).
For a chain
\(\mathbf X=(\varnothing=X_0<\cdots<X_k=E)\) of actual cyclic flats,
Ferroni's valuative form~\cite[Theorem~5.2]{Ferroni2023} of Hampe's
cyclic-chain expansion~\cite[Theorem~3.12]{Hampe2017} gives the coefficient
\[
 \lambda_{\mathbf X}
 =-\mu_{\mathsf C(\Zcyc)}(\mathbf X,\widehat1)
\]
for \(\Sch_{\mathbf X}\).
Panzer's factorization formula~\cite[Proposition~3.6]{Panzer2025} gives
\[
 \lambda_{\mathbf X}
 =
 \prod_{i=1}^{k}
   \bigl(-\mu_{\Zcyc}(X_{i-1},X_i)\bigr).
\]

We now collect the terms with a fixed block sequence \(\mathbf B\).
They all give matroids isomorphic to \(\Sch_{\mathbf B}\), so it remains
to sum their coefficients.  Put
\[
 c_i=-(A_{\Pcal}^{-1})_{B_{i-1},B_i}.
\]
Set \(F_0(\varnothing)=1\).  For \(Y\in B_i\), let \(F_i(Y)\) be the
total weight of chains with block sequence \(B_0,\ldots,B_i\) ending
at \(Y\).  We will show that this total is the same for every
\(Y\in B_i\).  Explicitly,
\[
 F_i(Y)=
 \sum_{\substack{X_0<\cdots<X_i=Y\\
                  X_j\in B_j\ (0\leq j\leq i)}}
 \prod_{j=1}^i\bigl(-\mu_{\Zcyc}(X_{j-1},X_j)\bigr).
\]
Separating a chain according to its penultimate flat gives
\[
 F_i(Y)=
 \sum_{\substack{X\in B_{i-1}\\X\subseteq Y}}
 F_{i-1}(X)\bigl(-\mu_{\Zcyc}(X,Y)\bigr).
\]
Lemma~\ref{lem:condensed-moebius} says that the sum of the M\"obius
factors on the right is \(c_i\), independently of \(Y\).
Thus, if \(F_{i-1}(X)=c_1\cdots c_{i-1}\) for all \(X\in B_{i-1}\),
then \(F_i(Y)=c_1\cdots c_i\) for all \(Y\in B_i\).  Since
\(F_0(\varnothing)=1\), induction gives
\[
 F_i(Y)=c_1\cdots c_i\qquad\text{for all }Y\in B_i.
\]
As \(B_k=\{E\}\), \(F_k(E)=c_1\cdots c_k\) is exactly the total
coefficient of the block sequence \(\mathbf B\).  Substitution into the
cyclic-chain expansion gives~\eqref{eq:Schubert-expansion}.
\end{proof}

\medskip
A \emph{covaluation} satisfies the subdivision relation without the
alternating dimension signs.  Equivalently, it is linear on indicator
functions of the relative interiors of matroid base
polytopes~\cite[Section~2.3]{Ferroni2023}.

\begin{corollary}[Covaluative Schubert expansion]
\label{cor:covaluative-expansion}
Let \(\Psi\) be a covaluative matroid invariant with values in an
abelian group, and let \(c(M)\) be the number of connected components
of \(M\).  Then
\[
 \Psi(M)=(-1)^{c(M)-1}
 \sum_{\zero=B_0<\cdots<B_k=\one}
 \left(
 \prod_{i=1}^{k}
   -(A_{\Pcal}^{-1})_{B_{i-1},B_i}
 \right)
 \Psi(\Sch_{\mathbf B}).
\]
In particular, a connected matroid has the same Schubert formula for
valuations and covaluations.
\end{corollary}

\begin{proof}
By the Euler-map relation~\cite[Corollary~5.3]{Ferroni2023},
\((-1)^{|E|-c(N)}\Psi(N)\) is valuative.  Apply
Theorem~\ref{thm:Schubert-expansion} and note that every loopless,
coloopless Schubert matroid in the sum is connected.  Finally, \(c(M)\)
is the least exponent of \(x\) in \(T_M(x,0)\), which is itself
determined by the condensed configuration.
\end{proof}

Known valuativeness results apply the theorem to the flag \(f\)-vector,
the Tutte polynomial and Derksen's \(\Ginv\)-invariant, the
base-polytope Ehrhart polynomial, the Kazhdan--Lusztig and
\(Z\)-polynomials, and the Chow polynomial
\(H_M(t)=\sum_i\dim\operatorname{CH}^i(M)t^i\)~\cite{DerksenFink2010,FerroniSchroeter2024}.
Corollary~\ref{cor:covaluative-expansion} applies to Speyer's
\(g\)-polynomial~\cite{Ferroni2023}.


\section{Schubert evaluations}\label{sec:evaluations}

The expansion reduces the computation of an invariant to its values on
Schubert matroids.  Let us recall how these can be obtained.

The Tutte polynomial can be computed using Kung's
recurrence~\cite[Lemma~4.7]{Kung2017}.  Explicit formulas are available for
Speyer's \(g\)-polynomial~\cite[Theorems~3.4 and~4.3]{Ferroni2023},
the Ehrhart polynomial~\cite[Theorem~1.1]{FanLi2024}, and the Chow
polynomial~\cite[Proposition~8.12]{FerroniSchroeter2024}.

For \(P\) and \(Z\), we can use their defining
recurrences~\cite{EPW2016,PXY2018}.  Schubert matroids are closed under
minors~\cite{Hampe2017}, so these computations remain within the
Schubert class.  Partition formulas are also known for the Catalan
subfamily~\cite{ChenLiYao2026}.


\section{Steiner \texorpdfstring{\(2\)}{2}-designs}

Let \(D\) be a Steiner system \(S(2,k,v)\), where \(2<k<v\); that is, a
design on \(v\) points whose blocks have \(k\) elements and in which
every pair of points lies in exactly one block.  Let \(M(D)\) be the
rank-\(3\) matroid whose rank-\(2\) flats are the blocks of \(D\).  Put
\[
 b=\frac{\binom v2}{\binom k2},
\]
the number of blocks of \(D\).
Recall that a rank-\(r\) matroid is \emph{paving} if every circuit has
cardinality at least \(r\).  The matroid \(M(D)\) is paving because every
pair of points is independent.  Every pair also belongs to a
three-element circuit, so \(M(D)\) is connected.  The \(b\) design
blocks are the only nontrivial cyclic flats and form one condensation
block \(B\), with labels
\((0,0),(k,2),(v,3)\).  Hence
\[
 A_{\Pcal}=
 \begin{pmatrix}1&1&1\\0&1&b\\0&0&1\end{pmatrix},
 \qquad
 A_{\Pcal}^{-1}=
 \begin{pmatrix}1&-1&b-1\\0&1&-b\\0&0&1\end{pmatrix}.
\]
There are only two block chains.  The chain \(\zero<\one\) gives
\(\Sch\bigl((0,0),(v,3)\bigr)=U_{3,v}\) with coefficient \(1-b\).  The
chain \(\zero<B<\one\) gives \(\Sch\bigl((0,0),(k,2),(v,3)\bigr)\) with
coefficient \(b\).
For every \(\Phi\) as in Theorem~\ref{thm:Schubert-expansion}, we have
\[
 \Phi(M(D))
 =b\,\Phi\bigl(\Sch((0,0),(k,2),(v,3))\bigr)-(b-1)\Phi(U_{3,v}).
\]
Corollary~\ref{cor:covaluative-expansion} gives the same identity for
the \(g\)-polynomial.  Evaluating these two Schubert matroids as in
Section~\ref{sec:evaluations} gives
\[
\begin{aligned}
 P_{M(D)}(t)&=1+(b-v)t,\\
 Z_{M(D)}(t)&=1+b(t+t^2)+t^3,\\
 H_{M(D)}(t)&=1+(b+1)t+t^2,
\end{aligned}
\]
and Ferroni's formula simplifies to
\[
\begin{aligned}
 g_{M(D)}(t)={}&
 \left[\binom{v-2}{2}-b\binom{k-1}{2}\right]t\\
 &+\left[(v-3)(v-4)-b(k-2)^2\right]t^2\\
 &+\left[\binom{v-4}{2}-b\binom{k-2}{2}\right]t^3.
\end{aligned}
\]
The four identities also follow from known paving-matroid formulas.
The \(P\)- and \(Z\)-formulas are special cases
of~\cite[Theorem~4.4]{FerroniNasrVecchi2023}.  The Chow and
\(g\)-formulas follow
from~\cite[Theorems~8.16 and~9.18]{FerroniSchroeter2024}.
Equivariant \(P\)-polynomials have also been computed explicitly for
several higher Steiner systems~\cite[Proposition~6.2]{KarnNasrProudfootVecchi2023}.

\begin{example}[The Janko--Tonchev design]
Janko and Tonchev constructed an \(S(2,7,175)\) with
\(b=725\) blocks~\cite{JankoTonchev1998}.  Its \(3\times3\) condensed
matrix gives, without using the individual blocks,
\[
\begin{aligned}
 P(t)&=1+550t,&
 Z(t)&=1+725t+725t^2+t^3,\\
 H(t)&=1+726t+t^2,&
 g(t)&=4003t+11287t^2+7285t^3.
\end{aligned}
\]
\end{example}


\section{The extended binary Golay matroid}

Let \(G_{24}\) be the rank-\(12\) column matroid of a generator matrix
for the extended binary Golay code.
The Mathieu group \(M_{24}\) has six orbits on its cyclic flats.  Label
the rows and columns of the condensed containment matrix by the common
cardinality and rank of the cyclic flats in each orbit.  Then
\[
 \bigl(A_{\Pcal}(B,C)\bigr)_{B,C}
 =
 \bordermatrix{
 &(0,0)&(8,7)&(12,10)&(12,11)&(16,11)&(24,12)\cr
 (0,0)&1&1&1&1&1&1\cr
 (8,7)&0&1&3&0&30&759\cr
 (12,10)&0&0&1&0&140&35420\cr
 (12,11)&0&0&0&1&0&2576\cr
 (16,11)&0&0&0&0&1&759\cr
 (24,12)&0&0&0&0&0&1\cr
 }.
\]
The top-column entry \(759\) in the second row is printed as \(75\)
in~\cite[Example~4(3)]{Eberhardt2014}.  It counts all \(759\) octads,
so the correct entry is \(759\).

Its inverse is
\[
 A_{\Pcal}^{-1}=
 \begin{pmatrix}
 1&-1&2&-1&-251&123003\\
 0&1&-3&0&390&-190509\\
 0&0&1&0&-140&70840\\
 0&0&0&1&0&-2576\\
 0&0&0&0&1&-759\\
 0&0&0&0&0&1
 \end{pmatrix}.
\]
Numbering the six condensation blocks \(0,\ldots,5\) in the displayed
order, we obtain the following nine nonzero Schubert coefficients.
\[
\begin{array}{c|r@{\qquad}c|r}
 \mathbf B&\lambda(\mathbf B)&\mathbf B&\lambda(\mathbf B)\\ \hline
 (0,5)&-123003 &(0,1,5)&190509\\
 (0,2,5)&141680 &(0,3,5)&2576\\
 (0,4,5)&190509 &(0,1,2,5)&-212520\\
 (0,1,4,5)&-296010 &(0,2,4,5)&-212520\\
 (0,1,2,4,5)&318780 &&
\end{array}
\]
For example, the chain \((0,1,4,5)\) has coefficient
\[
 \lambda(0,1,4,5)=1\cdot(-390)\cdot759=-296010,
\]
obtained by multiplying the negatives of the corresponding entries of
\(A_{\Pcal}^{-1}\).
The value of any valuative matroid invariant on \(G_{24}\) is therefore
a weighted sum of only these nine Schubert values.

For the Kazhdan--Lusztig, \(Z\)-, and Chow polynomials, this gives:
\[
\begin{aligned}
 P_{G_{24}}(t)={}&1+3311t+162656t^2+2419692t^3\\
                 &+6722716t^4+2923921t^5.
\end{aligned}
\]
\[
\begin{aligned}
 Z_{G_{24}}(t)={}&
 1+3335t+205436t^2+4285820t^3+31050690t^4\\
 &+95326605t^5+137174576t^6+95326605t^7\\
 &+31050690t^8+4285820t^9+205436t^{10}
   +3335t^{11}+t^{12}.
\end{aligned}
\]
\[
\begin{aligned}
 H_{G_{24}}(t)={}&1+2047093t+915803398t^2+36469915409t^3\\
 &+335015656738t^4+961738406039t^5+961738406039t^6\\
 &+335015656738t^7+36469915409t^8+915803398t^9\\
 &+2047093t^{10}+t^{11}.
\end{aligned}
\]

\begin{theorem}[Golay counterexample]
Neither \(P_{G_{24}}(t)\) nor \(Z_{G_{24}}(t)\) is real-rooted.
The extended binary Golay matroid is therefore a counterexample to
the real-rootedness conjectures for matroid Kazhdan--Lusztig polynomials
and \(Z\)-polynomials.
\end{theorem}

\begin{proof}
The displayed \(P\)-polynomial has
\[
 \operatorname{Disc}(P_{G_{24}})
 =
 -843\,715\,237\,941\,819\,667\,313\,354\,034\,040\,
 305\,367\,976\,443<0.
\]
For \(Z\), write
\[
 Z_{G_{24}}(t)=t^6h(t+t^{-1}),
\qquad
\begin{aligned}
 h(z)={}&z^6+3335z^5+205430z^4+4269145z^3\\
        &+30228955z^2+82485820z+75484066.
\end{aligned}
\]
An exact resultant calculation yields
\[
 \operatorname{Disc}(h)
 =
 -38\,099\,706\,872\,043\,872\,627\,476\,207\,222\,032\,620\,
 577\,668\,197\,052\,812\,043\,856<0.
\]
Thus \(P_{G_{24}}\) and \(h\) have non-real roots.  A non-real root
\(z\) of \(h\) lifts through \(t+t^{-1}=z\) to non-real roots of
\(Z_{G_{24}}\).  The two conclusions
contradict~\cite[Conjecture~3.2]{GPY2017}
and~\cite[Conjecture~5.1]{PXY2018}, respectively.
\end{proof}

\section*{Acknowledgments}

We thank Nicholas Proudfoot for helpful input and for making us aware
of the concurrent and independent work of Ronnie Cheng and Shurui Liu.

\printbibliography

\end{document}